\documentclass[11pt,a4paper,oneside]{article}

\usepackage{authblk}

\title{%
    Counting numbers that are divisible by\\ the product of their digits
}

\author[1]{Qizheng He}
\author[2]{Carlo Sanna\footnote{C.~Sanna is a member of GNSAGA of INdAM and of CrypTO, the group of Cryptography and Number~Theory of the Politecnico di Torino.}}
\affil[1]{%
    \texttt{hqztrue@sina.com}
    \vspace{5pt}
}%
\affil[2]{%
    Department of Mathematical Sciences, Politecnico di Torino\protect\\
    Corso Duca degli Abruzzi 24, 10129 Torino, Italy\protect\\
    \texttt{carlo.sanna@polito.it}
}

\date{} % no date

\usepackage{amsmath}
\usepackage{amssymb}
\usepackage{amsthm}
\usepackage{bm}
\usepackage{booktabs}
\usepackage{diagbox}
\usepackage{adjustbox}
\usepackage{placeins}
\usepackage[left=1.15in, right=1.15in, top=.72in, bottom=.72in]{geometry}
\usepackage[colorlinks=true]{hyperref}
\usepackage{color}
\usepackage{enumitem}
\setlist[enumerate]{label=(\roman*),labelindent=1em,itemsep=0.5em,topsep=0.5em}
\usepackage{graphicx}

\newtheorem{theorem}{Theorem}[section]

\newtheorem{lemma}[theorem]{Lemma}
\theoremstyle{remark}
\newtheorem{remark}{Remark}[section]

\newcommand{\high}{\textsf{hi}}
\newcommand{\mi}{\textsf{mi}}
\newcommand{\low}{\textsf{lo}}

\begin{document}

\maketitle

\begin{abstract}
    Let $b \geq 3$ be a positive integer.
    A natural number is said to be a \emph{base-$b$ Zuckerman number} if it is divisible by the product of its base-$b$ digits.
    Let $\mathcal{Z}_b(x)$ be the set of base-$b$ Zuckerman numbers that do not exceed $x$, and assume that $x \to +\infty$.

    First, we prove an upper bound of the form $|\mathcal{Z}_b(x)| < x^{z_b^{+} + o(1)}$, where $z_b^{+} \in (0,1)$ is an effectively computable constant.
    In particular, we have that $z_{10}^+ = 0.665{\scriptstyle\ldots}$, which improves upon the previous upper bound $|\mathcal{Z}_{10}(x)| < x^{0.717}$ due to Sanna.
    Moreover, we prove that $|\mathcal{Z}_{10}(x)| > x^{0.204}$, which improves upon the previous lower bound $|\mathcal{Z}_{10}(x)| > x^{0.122}$, due to De~Koninck and Luca.

    Second, we provide a heuristic suggesting that $|\mathcal{Z}_b(x)| = x^{z_b + o(1)}$, where $z_b \in (0,1)$ is an effectively computable constant.
    In particular, we have that $z_{10} = 0.419{\scriptstyle\ldots}$.

    Third, we provide algorithms to count, respectively enumerate, the elements of $\mathcal{Z}_b(x)$, and we determine their complexities.
    Implementing one of such counting algorithms, we computed $|\mathcal{Z}_b(x)|$ for $b=3,\dots,12$ and large values of $x$ (depending on $b$), and we showed that the results are consistent with our heuristic.
\end{abstract}


% Number theory
%   Elementary number theory
%     11A63 Radix representation; digital problems
%   Multiplicative number theory
%     11N25 Distribution of integers with specified multiplicative constraints
%   Computational number theory
%     11Y16 Number-theoretic algorithms; complexity
%     11Y55 Calculation of integer sequences

\section{Introduction}

Let $b \geq 2$ be a positive integer.
Natural numbers that satisfy special arithmetic constraints in terms of their base-$b$ digits have been studied by several authors.

For instance, a natural number is said to be a \emph{base-$b$ Niven number} if it is divisible by the sum of its base-$b$ digits.
Cooper and Kennedy proved that the set of base-$10$ Niven numbers has natural density equal to zero~\cite{CK84}, provided a lower bound for its counting function~\cite{MR0809074}, and gave an asymptotic formula for the counting function of base-$10$ Niven numbers with a fixed sum of digits~\cite{MR0938592}.
Later, Vardi~\cite[Sec.~2.3]{MR1150054} gave stronger upper and lower bounds for the counting function of base-$10$ Niven numbers.
Then, De~Koninck, Doyon, and K\'atai~\cite{MR1957109}, and (independently) Mauduit, Pomerance, and S\'ark\"ozy~\cite{MR2166377}, proved that the number of base-$b$ Niven numbers not exceeding $x$ is asymptotic to $c_b x / \! \log x$, as $x \to +\infty$, where $c_b > 0$ is an explicit constant (see~\cite{MR2605530} for a generalization).
Furthermore, De~Koninck and Doyon~\cite{MR1988644} studied large gaps between base-$b$ Niven numbers, and De~Koninck, Doyon, and K\'{a}tai~\cite{MR2439527} provided an asymptotic formula for the number of $r$-tuples of consecutive base-$b$ Niven numbers not exceeding $x$.
Conditionally to Hooley's Riemann hypothesis, Sanna~\cite{MR4338466} proved that every sufficiently large positive integer is the sum of a bounded number (depending only on $b$) of base-$b$ Niven numbers.

A natural number is said to be a \emph{base-$b$ Zuckerman number} if it is divisible by the product of its base-$b$ digits.
(If one requires divisibility by the product of nonzero base-$b$ digits, then one gets the \emph{base-$b$ prodigious numbers}~\cite{MR4503211}.)
Note that the base-$2$ Zuckerman numbers are simply the \emph{Mersenne numbers} $2^k - 1$, where $k=1,2,\dots$, which are also called the \emph{base-$2$ repunits}.
Hence, hereafter, we assume that $b \geq 3$.
Zuckerman numbers are less studied than their additive counterpart of Niven numbers, but there are some results.
For all $x \geq 1$, let $\mathcal{Z}_b(x)$ be the sets of base-$b$ Zuckerman numbers not exceeding $x$.
De~Koninck and Luca~\cite{MR2298113,MR3734412} proved upper and lower bounds for the cardinality of $\mathcal{Z}_{10}(x)$.
Then Sanna~\cite{MR4181552} gave an upper bound for the cardinality of $\mathcal{Z}_b(x)$, which for $b = 10$ improves the result of De~Koninck and Luca.
More precisely, putting together the bounds of De~Koninck--Luca and Sanna, we have that
\begin{equation}\label{equ:previous-bounds}
    x^{0.122} < |\mathcal{Z}_{10}(x)| < x^{0.717}
\end{equation}
for all sufficiently large $x$.

Our contribution is the following.
Assume that $x \to +\infty$.
In Section~\ref{sec:upper-bound}, we prove an upper bound of the form $|\mathcal{Z}_b(x)| < x^{z_b^{+} + o(1)}$, where $z_b^{+} \in (0,1)$ is an effectively computable constant.
In particular, we have that $z_{10}^+ = 0.665{\scriptstyle\ldots}$, which improves upon the upper bound of~\eqref{equ:previous-bounds}.
In Section~\ref{sec:lower-bound-10}, we prove the lower bound $|\mathcal{Z}_b(x)| > x^{0.204}$, which improves upon the lower bound of~\eqref{equ:previous-bounds}.
In Section~\ref{sec:heuristic}, we provide a heuristic suggesting that $|\mathcal{Z}_b(x)| = x^{z_b + o(1)}$, where $z_b \in (0,1)$ is an effectively computable constant.
In Section~\ref{sec:algorithms}, we provide algorithms to count, respectively to enumerate, the elements of $\mathcal{Z}_b(x)$.
Our best (general) counting algorithm has complexity $x^{z_b^* + o(1)}$, where $z_b^* \in (0, 1)$ is an effectively computable constant.
We collected the values of $z_b$, $z_b^+$, $z_b^*$, for $b=3,\dots,12$, in Table~\ref{tab:exponents}.
For $b = 10$, we also provide a counting, respectively enumeration, algorithm with complexity $x^{0.3794}$, respectively $x^{0.3794} + |\mathcal{Z}_{10}(x)|$.
In particular, assuming the previous heuristic, this enumeration algorithm is (asymptotically) optimal.
Finally, Qizheng~He implemented one of the counting algorithms in \texttt{C++} (the implementation is freely available on GitHub~\cite{Repo}).
Using such implementation, we computed the number of base-$b$ Zuckerman numbers with exactly $N$ digits, for $b=3,\dots,12$ and large values of $N$ (about 80 hours on a consumer laptop).
The results are collected in Table~\ref{tab:data} and support our heuristic, see Table~\ref{tab:errors} (and Lemma~\ref{lem:from-N-to-x} for the justification of the comparison).
(The terms in Table~\ref{tab:data} with $b=10$ and $N \leq 16$ were already computed by Giovanni Resta, see his comment to sequence $\textsf{A007602}$ of OEIS~\cite{OEIS}.) 

\begin{table}[ht]
    \centering
    \adjustbox{max width=\textwidth}{
    \begin{tabular}{c|rrrrrrrrrr}
    \toprule
    $b$ & \multicolumn{1}{c}{$3$} & \multicolumn{1}{c}{$4$} & \multicolumn{1}{c}{$5$} & \multicolumn{1}{c}{$6$} & \multicolumn{1}{c}{$7$} & \multicolumn{1}{c}{$8$} & \multicolumn{1}{c}{$9$} & \multicolumn{1}{c}{$10$} & \multicolumn{1}{c}{$11$} & \multicolumn{1}{c}{$12$} \\
    $z_b$ & $.3690{\scriptstyle\ldots}$ & $.2075{\scriptstyle\ldots}$ & $.4560{\scriptstyle\ldots}$ & $.3727{\scriptstyle\ldots}$ & $.4604{\scriptstyle\ldots}$ & $.2483{\scriptstyle\ldots}$ & $.3625{\scriptstyle\ldots}$ & $.4197{\scriptstyle\ldots}$ & $.4481{\scriptstyle\ldots}$ & $.3537{\scriptstyle\ldots}$ \\
    $z_b^+$ & $.5257{\scriptstyle\ldots}$ & $.4024{\scriptstyle\ldots}$ & $.6634{\scriptstyle\ldots}$ & $.5948{\scriptstyle\ldots}$ & $.6885{\scriptstyle\ldots}$ & $.4988{\scriptstyle\ldots}$ & $.6081{\scriptstyle\ldots}$ & $.6657{\scriptstyle\ldots}$ & $.6977{\scriptstyle\ldots}$ & $.6130{\scriptstyle\ldots}$ \\
    $z_b^*$ & $.4318{\scriptstyle\ldots}$ & $.3018{\scriptstyle\ldots}$ & $.4361{\scriptstyle\ldots}$ & $.3866{\scriptstyle\ldots}$ & $.4559{\scriptstyle\ldots}$ & $.3304{\scriptstyle\ldots}$ & $.4017{\scriptstyle\ldots}$ & $.4416{\scriptstyle\ldots}$ & $.4653{\scriptstyle\ldots}$ & $.4068{\scriptstyle\ldots}$ \\
    \bottomrule
    \end{tabular}
    }
    \caption{Exponents $z_b$, $z_b^+$, $z_b^*$ for $b=3,\dots,12$.}
    \label{tab:exponents}
\end{table}
\vspace{-2em}
\section{Notation}

Throughout the rest of the paper, let $b \geq 3$ be a fixed integer.
For the sake of brevity, we say ``digits'', ``representation'', and ``Zuckerman numbers'' instead of ``base-$b$ digits'', ``base-$b$ representation'', and ``base-$b$ Zuckerman numbers'', respectively.
We employ the Landau--Bachmann ``Big Oh'' and ``little oh'' notations $O$ and $o$ with their usual meanings.
In particular, the implied constants in $O$, and how fast expressions like $o(1)$ go to zero, may depend on $b$.
We let $\mathbb{N} := \{0,1,\dots\}$ be the set of nonnegative integers, and we let $|\mathcal{S}|$ denote the cardinality of every finite set $\mathcal{S}$.
For every integer $n \geq 1$, let $\mathcal{D}_b(n)$ be the set of the digits of $n$, and let $p_b(n) := \prod_{d = 0}^{b - 1} d^{w_{b,d}(n)}$ be the product of the digits of $n$, where $w_{b,d}(n)$ denotes the number of times that the digit $d$ appears in the representation of $n$.
Let $\mathcal{Z}_b$ be the set of Zuckerman numbers, and let $\mathcal{Z}_{b,N}$ be the set of Zuckerman numbers with exactly $N$ digits.
%\newpage
\begin{table}[ht]
    \centering
    \begin{adjustbox}{max width=0.9\textwidth}
    \begin{tabular}{rrrrrrrrrrr}
    \toprule
    \diagbox{$N$}{$b$} & $3$ & $4$ & $5$ & $6$ & $7$ & $8$ & $9$ & $10$ & $11$ & $12$ \\
    \midrule
    $1$ & $2$ & $3$ & $4$ & $5$ & $6$ & $7$ & $8$ & $9$ & $10$ & $11$\\
    $2$ & $2$ & $2$ & $3$ & $4$ & $3$ & $4$ & $5$ & $5$ & $5$ & $6$\\
    $3$ & $4$ & $4$ & $14$ & $8$ & $23$ & $15$ & $18$ & $20$ & $33$ & $21$\\
    $4$ & $6$ & $7$ & $10$ & $20$ & $29$ & $9$ & $33$ & $40$ & $63$ & $43$\\
    $5$ & $9$ & $6$ & $42$ & $27$ & $96$ & $38$ & $107$ & $117$ & $224$ & $107$\\
    $6$ & $10$ & $8$ & $78$ & $55$ & $203$ & $49$ & $191$ & $285$ & $645$ & $222$\\
    $7$ & $14$ & $16$ & $184$ & $109$ & $533$ & $78$ & $518$ & $747$ & $2000$ & $544$\\
    $8$ & $33$ & $18$ & $385$ & $188$ & $1295$ & $163$ & $914$ & $1951$ & $5411$ & $1213$\\
    $9$ & $46$ & $22$ & $795$ & $364$ & $3299$ & $294$ & $1959$ & $5229$ & $16532$ & $2718$\\
    $10$ & $43$ & $36$ & $1570$ & $653$ & $7630$ & $376$ & $4903$ & $13493$ & $45464$ & $6267$\\
    $11$ & $72$ & $38$ & $3208$ & $1095$ & $19130$ & $631$ & $11129$ & $35009$ & $135967$ & $13738$\\
    $12$ & $171$ & $53$ & $6411$ & $2076$ & $43687$ & $1246$ & $22161$ & $91792$ & $393596$ & $31483$\\
    $13$ & $211$ & $77$ & $13741$ & $3866$ & $111255$ & $1966$ & $50391$ & $239791$ & $1161371$ & $71482$\\
    $14$ & $252$ & $96$ & $29200$ & $7373$ & $276967$ & $3408$ & $116777$ & $628412$ & $3406099$ & $160109$\\
    $15$ & $428$ & $129$ & $60864$ & $14622$ & $690189$ & $6038$ & $261725$ & $1643144$ & $10012223$ & $366977$\\
    $16$ & $728$ & $177$ & $126080$ & $27972$ & $1710625$ & $8291$ & $578324$ & $4314987$ & $29355933$ & $845908$\\
    $17$ & $986$ & $237$ & $263060$ & $53201$ & $4124693$ & $13470$ & $1276433$ & $11319722$ & $86022519$ &  \\
    $18$ & $1400$ & $317$ & $545025$ & $103132$ & $10097943$ & $28419$ & $2851060$ & $29713692$ & $251993074$ &  \\
    $19$ & $2214$ & $425$ & $1137646$ & $203051$ & $24765215$ & $46596$ & $6310957$ & $78042616$ & $737799286$ &  \\
    $20$ & $3450$ & $558$ & $2371769$ & $398775$ & $60708268$ & $71497$ & $13886129$ & $204939760$ &   &  \\
    $21$ & $5007$ & $772$ & $4946854$ & $774024$ & $148622249$ & $126490$ & $30753950$ & $538453205$ &   &  \\
    $22$ & $7370$ & $997$ & $10296601$ & $1506714$ & $364176274$ & $198722$ & $68293912$ & $1414773364$ &   &  \\
    $23$ & $11234$ & $1305$ & $21454503$ & $2915442$ & $894674969$ & $320763$ & $151573306$ &   &   &  \\
    $24$ & $16981$ & $1817$ & $44678532$ & $5658200$ & $2204890644$ & $603722$ &   &   &   &  \\
    $25$ & $25324$ & $2305$ & $93110027$ & $10999574$ & $5390633926$ & $1015093$ &   &   &   &  \\
    $26$ & $37716$ & $3096$ & $193971630$ & $21369791$ &   & $1585495$ &   &   &   &  \\
    $27$ & $56757$ & $4164$ & $404103162$ & $41626279$ &   & $2717026$ &   &   &   &  \\
    $28$ & $85493$ & $5495$ & $841843065$ & $81172184$ &   &   &   &   &   &  \\
    $29$ & $127774$ & $7402$ & $1753948967$ & $158009860$ &   &   &   &   &   &  \\
    $30$ & $191665$ & $9936$ & $3653927956$ & $307539610$ &   &   &   &   &   &  \\
    $31$ & $287481$ & $13013$ & $7612395846$ & $598683507$ &   &   &   &   &   &  \\
    $32$ & $431622$ & $17308$ &   &   &   &   &   &   &   &  \\
    $33$ & $646816$ & $23372$ &   &   &   &   &   &   &   &  \\
    $34$ & $970475$ & $31037$ &   &   &   &   &   &   &   &  \\
    $35$ & $1455724$ & $41399$ &   &   &   &   &   &   &   &  \\
    $36$ & $2183782$ & $55034$ &   &   &   &   &   &   &   &  \\
    $37$ & $3275092$ & $73086$ &   &   &   &   &   &   &   &  \\
    $38$ & $4914274$ & $98142$ &   &   &   &   &   &   &   &  \\
    $39$ & $7371941$ & $130591$ &   &   &   &   &   &   &   &  \\
    $40$ & $11057697$ & $173916$ &   &   &   &   &   &   &   &  \\
    $41$ & $16586242$ & $232253$ &   &   &   &   &   &   &   &  \\
    $42$ & $24880345$ & $309102$ &   &   &   &   &   &   &   &  \\
    $43$ & $37318948$ & $412940$ &   &   &   &   &   &   &   &  \\
    $44$ & $55979205$ & $549336$ &   &   &   &   &   &   &   &  \\
    $45$ & $83963507$ & $733783$ &   &   &   &   &   &   &   &  \\
    $46$ & $125950398$ & $978893$ &   &   &   &   &   &   &   &  \\
    $47$ & $188921345$ & $1305037$ &   &   &   &   &   &   &   &  \\
    $48$ & $283385733$ & $1738126$ &   &   &   &   &   &   &   &  \\
    $49$ & $425085179$ & $2319219$ &   &   &   &   &   &   &   &  \\
    $50$ & $637608602$ & $3091664$ &   &   &   &   &   &   &   &  \\
    $51$ & $956428288$ & $4119790$ &   &   &   &   &   &   &   &  \\
    $52$ & $1434628060$ &   &   &   &   &   &   &   &   &  \\
    $53$ & $2152013870$ &   &   &   &   &   &   &   &   &  \\
    $54$ & $3227959147$ &   &   &   &   &   &   &   &   &  \\
    $55$ & $4841970543$ &   &   &   &   &   &   &   &   &  \\
    $56$ & $7262855061$ &   &   &   &   &   &   &   &   &  \\
    $57$ & $10894279904$ &   &   &   &   &   &   &   &   &  \\
    $58$ & $16341567376$ &   &   &   &   &   &   &   &   &  \\
    $59$ & $24512322843$ &   &   &   &   &   &   &   &   &  \\
    \bottomrule
    \end{tabular}
    \end{adjustbox}
    \caption{The number of base-$b$ Zuckerman numbers with $N$ digits in base $b$.}
    \label{tab:data}
\end{table}

\begin{table}[h!]
    \centering
    \adjustbox{scale=0.8}{
    \begin{tabular}{rrrrrr}
    \toprule
    $b$ & $N$ & $|\mathcal{Z}_{b,N}|$ & $\widetilde{z}_{b,N}$ & $z_b$ & error \\
    \midrule
    $3$ & $59$ & $24512322843$ & $.36907025$ & $.36907024$ & $1.7{\scriptstyle\ldots} \mbox{\tiny$\times$} 10^{-8}$ \\
    $4$ & $51$ & $4119790$ & $.21543273$ & $.20751874$ & $3.8{\scriptstyle\ldots} \mbox{\tiny$\times$} 10^{-2}$ \\
    $5$ & $31$ & $7612395846$ & $.45604067$ & $.45604068$ & $2.8{\scriptstyle\ldots} \mbox{\tiny$\times$} 10^{-8}$ \\
    $6$ & $31$ & $598683507$ & $.36385650$ & $.37272266$ & $2.3{\scriptstyle\ldots} \mbox{\tiny$\times$} 10^{-2}$ \\
    $7$ & $25$ & $5390633926$ & $.46061589$ & $.46049815$ & $2.5{\scriptstyle\ldots} \mbox{\tiny$\times$} 10^{-4}$ \\
    $8$ & $27$ & $2717026$ & $.26387156$ & $.24839536$ & $6.2{\scriptstyle\ldots} \mbox{\tiny$\times$} 10^{-2}$ \\
    $9$ & $23$ & $151573306$ & $.37273465$ & $.36252164$ & $2.8{\scriptstyle\ldots} \mbox{\tiny$\times$} 10^{-2}$ \\
    $10$ & $22$ & $1414773364$ & $.41594031$ & $.41978534$ & $9.1{\scriptstyle\ldots} \mbox{\tiny$\times$} 10^{-3}$ \\
    $11$ & $19$ & $737799286$ & $.44818212$ & $.44816395$ & $4.0{\scriptstyle\ldots} \mbox{\tiny$\times$} 10^{-5}$ \\
    $12$ & $16$ & $845908$ & $.34327662$ & $.35378177$ & $2.9{\scriptstyle\ldots} \mbox{\tiny$\times$} 10^{-2}$ \\
\bottomrule
\end{tabular}
}
\caption{Here $\mathcal{Z}_{b,N}$ is the set of base-$b$ Zuckerman numbers with $N$ digits, $\widetilde{z}_{b,N} := \log\!|\mathcal{Z}_{b,N}| /\! \log(b^N)$, and the error is equal to $|\widetilde{z}_{b,N} - z_b| / z_b$.}
\label{tab:errors}
\end{table}
\FloatBarrier
\section{Preliminaries}

This section collects some preliminary lemmas needed in subsequent proofs.

\subsection{Zuckerman numbers}

The next lemma regards two simple, but useful, facts.

\begin{lemma}\label{lem:simple-facts}
	Let $n \in \mathcal{Z}_b$.
	Then
	\begin{enumerate}
		\item\label{ite:simple-facts:1} $0 \notin \mathcal{D}_b(n)$;
		\item\label{ite:simple-facts:2} $b$ does not divide $p_b(n)$.
	\end{enumerate}
\end{lemma}
\begin{proof}
	Since $n \in \mathcal{Z}_b$, we have that $p_b(n)$ divides $n$.
	If $0 \in \mathcal{D}_b(n)$, then $p_b(n) = 0$, and so $n = 0$, which is impossible.
	Thus~\ref{ite:simple-facts:1} follows.
	If $b$ divides $p_b(n)$, then $b$ divides $n$, so that the least significant digit of $n$ is equal to zero, which is impossible by~\ref{ite:simple-facts:1}.
	Thus~\ref{ite:simple-facts:2} follows.
\end{proof}

We need to define some families of sets of digits and some related quantities.
Let $\Omega_b$ be the family of all subsets $\mathcal{D} \subseteq \{1, \dots, b-1\}$ such that $b$ does not divide $d_1 \cdots d_k$ for all $k \geq 1$ and $d_1, \dots, d_k \in \mathcal{D}$.
Moreover, let $\Omega_b^*$ be the family of all $\mathcal{D} \in \Omega_b$ such that there exists no $\mathcal{D}^\prime \in \Omega_b$ with $\mathcal{D} \subsetneq \mathcal{D}^\prime$.
(A more explicit description of $\Omega_b^*$ is given by Remark~\ref{rmk:Omega-b-star}.)
Note that every $\mathcal{D} \in \Omega_b^*$ contains each $d \in \{1, \dots, b - 1\}$ with $\gcd(b, d) = 1$.
In particular, since $b \geq 3$, we have that $|\mathcal{D}| \geq 2$.

For every $k$ pairwise distinct digits $d_1, \dots, d_k \in \{1, \dots, b-1\}$, if there exist integers $v_1, \dots, v_k \geq 1$ such that $b$ divides $d_1^{v_1} \cdots d_k^{v_k}$ then let $V_b(d_1, \dots, d_k)$ be the minimum possible value of $\max(v_1, \dots, v_k)$, otherwise let $V_b(d_1, \dots, d_k) := 0$.
Furthermore, let $W_b$ be the maximum of $V_b(d_1, \dots, d_k)$ as $d_1, \dots, d_k$ range over all the possible $k$ pairwise distinct digits in $\{1, \dots, b-1\}$.
For each $\mathcal{D} \in \Omega_b^*$, let $\mathcal{Z}_{b,\mathcal{D}}$ be the set of all $n \in \mathcal{Z}_b$ such that $d \in \mathcal{D}_b(n) \setminus \mathcal{D}$ implies $w_{b,d}(n) < W_b$.
Also, for every integer $N \geq 1$, put $\mathcal{Z}_{b,\mathcal{D},N} := \mathcal{Z}_{b,\mathcal{D}} \cap {[b^{N-1}, b^N)}$.

\begin{lemma}\label{lem:Zb-as-union-of-ZbD}
	We have that $\mathcal{Z}_b = \bigcup_{\mathcal{D} \in \Omega_b^*} \mathcal{Z}_{b,\mathcal{D}}$.
\end{lemma}
\begin{proof}
	Since $\mathcal{Z}_{b,\mathcal{D}} \subseteq \mathcal{Z}_b$ for every $\mathcal{D} \in \Omega_b^*$, it suffices to prove that $\mathcal{Z}_b \subseteq \bigcup_{\mathcal{D} \in \Omega_b^*} \mathcal{Z}_{b,\mathcal{D}}$.
	Let $n \in \mathcal{Z}_b$.
	We have to prove that there exists $\mathcal{D} \in \Omega_b^*$ such that for each $d \in \mathcal{D}_b(n)$ we have that either $d \in \mathcal{D}$ or $w_{b,d}(n) < W_b$.
	Note that it suffices to prove the existence of $\mathcal{D} \in \Omega_b$, since every set in $\Omega_b$ is a subset of a set in $\Omega_b^*$.
	
	If $\mathcal{D}_b(n) \in \Omega_b$ then the claim is obvious.
	Hence, assume that $\mathcal{D}_b(n) \notin \Omega_b$.
	Thus there exist $k$ pairwise distinct digits $d_1, \dots, d_k \in \mathcal{D}_b(n)$ and integers $v_1, \dots, v_k \geq 1$ such that $b$ divides $d_1^{v_1} \cdots d_k^{v_k}$.
	In particular, without loss of generality, we can assume that $\max(v_1, \dots, v_k) = V_b(d_1, \dots, d_k)$.
	Consequently, we have that $\max(v_1, \dots, v_k) \leq W_b$.
	
	If $w_{b,d_i}(n) \geq W_b$ for each $i \in \{1, \dots, k\}$, then we get that $b$ divides $d_1^{w_{b,d_1}(n)} \cdots d_k^{w_{b,d_k}(n)}$, which in turn divides $p_b(n)$.
	But by Lemma~\ref{lem:simple-facts}\ref{ite:simple-facts:2} this is impossible, since $n \in \mathcal{Z}_b$.
	
	Therefore, there exists $i_1 \in \{1, \dots, k\}$ such that $w_{b,d_{i_1}}(n) < W_b$.
	At this point, we can repeat the previous reasonings with $\mathcal{D}^\prime := \mathcal{D}_b(n) \setminus \{d_{i_1}\}$ in place of $\mathcal{D}_b(n)$.
	If $\mathcal{D}^\prime \in \Omega_b$ then the claim follows.
	If $\mathcal{D}^\prime \notin \Omega_b$, then there exist $k^\prime$ pairwise distinct digits in $\mathcal{D}^\prime$ such that a certain product of them is divisible by $b$, \dots and so on.
	
	It is clear that this procedure terminates after at most $|\mathcal{D}_b(n)|$ steps (note that $\varnothing \in \Omega_b$), thus producing the desidered $\mathcal{D} \in \Omega_b$.
\end{proof}

\begin{remark}\label{rmk:Omega-b-star}
	It follows easily that the elements of $\Omega_b^*$ are the sets
	\begin{equation*}
		\big\{d \in \{1, \dots, b - 1\} : p \nmid d \big\} ,
	\end{equation*}
	where $p$ runs over the prime factors of $b$.
	In particular, if $b$ is a prime number then the only element of $\Omega_b^*$ is $\{1, \dots, b-1\}$.
\end{remark}

\subsection{Entropy and multinomial sums}

For all real numbers $x_1, \dots, x_k \in [0,1]$ such that $\sum_{i=1}^k x_i = 1$, define the \emph{entropy}
\begin{equation*}
    H(x_1, \dots, x_k) := - \sum_{i \,=\, 1}^k x_i \log x_i ,
\end{equation*}
with the usual convention that $0 \log 0 := 0$.

\begin{lemma}\label{lem:multinomial}
    Let $N_1, \dots, N_k \geq 0$ be integers and put $N := N_1 + \cdots + N_k$.
    Then
    \begin{equation*}
        \frac{N!}{N_1! \cdots N_k!} = \exp\!\left(H\!\left(\frac{N_1}{N}, \dots, \frac{N_k}{N}\right) N + O_k(\log N)\right) ,
    \end{equation*}
    as $N \to +\infty$.
\end{lemma}
\begin{proof}
    The claim follows easily from Stirling's formula in the form
    \begin{equation*}
        \log n! = n \log n - n + O(\log n) ,
    \end{equation*}
    as $n \to +\infty$.
\end{proof}

\begin{lemma}\label{lem:max-entropy}
    Let $a_1, \dots, a_k \geq 0$ and $c$ be real numbers such that
    \begin{equation}\label{equ:min-mean-condition}
        \min\{a_1,\dots,a_k\} < c < \frac1{k}\sum_{i \,=\, 1}^k a_i .
    \end{equation}
    Then the equation
    \begin{equation}\label{equ:max-entropy-equation}
        \sum_{i \,=\, 1}^k (a_i - c) e^{a_i \lambda} = 0
    \end{equation}
    has a unique solution $\lambda \in (-\infty, 0)$.
    Moreover, the maximum of $H(x_1, \dots, x_k)$ under the constraints $\sum_{i=1}^k x_i = 1$ and $\sum_{i = 1}^k a_i x_i \leq c$ is equal to
    \begin{equation*}
        H_{\max} := -c \lambda + \log\!\left(\sum_{i \,=\, 1}^k e^{a_i \lambda} \right) ,
    \end{equation*}
    and it is achieved if and only if $x_i = e^{a_i \lambda} / \sum_{j=1}^k e^{a_j \lambda}$ for $i=1,\dots,k$.
\end{lemma}
\begin{proof}
    First, note that $a_1, \dots, a_k$ are not all equal, otherwise~\eqref{equ:min-mean-condition} would not be satisfied.
    For every real number $t$, define
    \begin{equation*}
        F(t) := \frac{\sum_{i = 1}^k a_i e^{a_i t}}{\sum_{i = 1}^k e^{a_i t}} - c .
    \end{equation*}
    We have that
    \begin{equation*}
        F^\prime(t) = \frac{\left(\sum_{i = 1}^k a_i^2 e^{a_i t}\right)\left(\sum_{i = 1}^k e^{a_i t}\right) - \left(\sum_{i = 1}^k a_i e^{a_i t}\right)^2}{\left(\sum_{i = 1}^k e^{a_i t}\right)^2} > 0 ,
    \end{equation*}
    where we applied the Cauchy--Schwarz inequality, which is strict since $a_1, \dots, a_k$ are not all equal.
    Furthermore, by \eqref{equ:min-mean-condition} we have that
    \begin{equation*}
        \lim_{t \,\to\, -\infty} F(t) = \min\{a_1, \dots, a_k\} - c < 0 \quad\text{ and }\quad F(0) = \frac1{k} \sum_{i \,=\, 1}^k a_i - c > 0 .
    \end{equation*}
    Therefore, there exists a unique $\lambda \in (-\infty, 0)$ such that $F(\lambda) = 0$, which is equivalent to~\eqref{equ:max-entropy-equation}.

    The rest of the lemma is a standard application of the method of Lagrange multipliers (cf.\ \cite[p.\ 328, Example~3]{MR3363684}).
    We only sketch the proof.
    After introducing the slack variable $x_{k+1}$, the problem becomes to maximize $H = -\sum_{i = 1}^k x_i \log x_i$ under the constrains $G_1 := \sum_{i = 1}^k x_i - 1 = 0$ and $G_2 := \sum_{i = 1}^k a_i x_i + x_{k+1}^2 - c = 0$.
    Let $L := H + \lambda_1 G_1 + \lambda_2 G_2$, where $\lambda_1, \lambda_2 \in \mathbb{R}$ are the Lagrange multipliers.
    Note that the $2 \times (k + 1)$ matrix whose entry of the $i$th row and $j$th colum is equal to $\partial G_i / \partial x_j$ has full rank.
    Hence, by Lagrange's theorem~\cite[p.~326, Theorem]{MR3363684}, the constrained local extrema of $H$ are obtained by solving the system of $k+2$ equations $\partial L / \partial x_i = 0$ ($i=1,\dots,k$) and $\partial L / \partial \lambda_j = 0$ ($j=1,2$).
    This system has the unique solution $x_i = e^{\lambda_1 + a_i \lambda_2 - 1}$ ($i=1,\dots,k$), $x_{k+1} = 0$, $\lambda_1 = 1 - \log\!\big(\sum_{i=1}^k e^{a_i \lambda_2}\big)$, and $\lambda_2 = \lambda$.
    We point out that in the proof of this last claim the condition $c < \sum_{i=1}^k a_i / k$ is used to show that $\lambda_2 \neq 0$, which in turn implies that $x_{k+1} = 0$.
    Finally, that the local extremum is a maximum follows easily from the study of the Hessian matrix of $H$.
\end{proof}

\begin{remark}\label{rmk:converse}
    Somehow conversely to the first part of Lemma~\ref{lem:max-entropy},
    it is not difficult to prove that if $a_1, \dots, a_k \geq 0$, $c$, and $\lambda$ are real numbers satifying~\eqref{equ:max-entropy-equation}, then \eqref{equ:min-mean-condition} holds.
\end{remark}

\begin{lemma}\label{lem:multinomial-sum}
    Let $a_1, \dots, a_k, c$ and $H_{\max}$ be as in Lemma~\ref{lem:max-entropy}, let $h \geq k$ be an integer, and let $C \geq 0$ be a real number.
    Then we have that
    \begin{equation}\label{equ:multinomial-sum0}
        \sum \frac{N!}{N_1! \cdots N_h!} = \exp\!\big((H_{\max} + o(1)) N \big) ,
    \end{equation}
    as $N \to +\infty$, where $N := N_1 + \cdots + N_h$, the sum is over all integers $N_1, \dots, N_h \geq 0$ such that $\sum_{i = 1}^k a_i N_i \leq c N$ and $N_i \leq C$ for each integer $i \in {(k,h]}$, and the $o(1)$ depends only on $a_1, \dots, a_k, c, h, C$.
\end{lemma}
\begin{proof}
    Let $N_- := \sum_{i=1}^k N_i$ and $N_+ := \sum_{i=k+1}^h N_i$.
    Then we have that
    \begin{align}\label{equ:multinomial-sum1}
        \sum \frac{N!}{N_1! \cdots N_h!} &= \sum_{\sum_{i = 1}^k a_i N_i \,\leq\, c N} \frac{N_-!}{\prod_{i = 1}^k (N_i!)} \, \sum_{\substack{N_i \,\leq\, C \\[1pt] i \,=\, k+1,\dots,h}} \frac{\prod_{n=1}^{N_+} (N_- + n)}{\prod_{i = k + 1}^h (N_i!)} \nonumber\\
        &= \sum_{\sum_{i = 1}^k a_i N_i \,\leq\, c N} \frac{N_-!}{\prod_{i = 1}^k (N_i!)} \; \exp\!\big(O_{C,h}(\log N)\big) ,
    \end{align}
    as $N \to +\infty$, where we employed the facts that $N_i \leq C$ for each integer $i \in {(k,h]}$, and consequently $N_+ \leq Ch$.
    Furthermore, letting
    \begin{equation*}
        M :=  \max_{\sum_{i = 1}^k a_i N_i \,\leq\, c N} \frac{N_-!}{\prod_{i = 1}^k (N_i!)} ,
    \end{equation*}
    we have that
    \begin{equation}\label{equ:multinomial-sum2}
        \sum_{\sum_{i = 1}^k a_i N_i \,\leq\, c N} \frac{N_-!}{\prod_{i = 1}^k (N_i!)} =  M \exp\!\big(O_k(\log N)\big) ,
    \end{equation}
    as $N \to +\infty$, where we exploited the fact that the sum in~\eqref{equ:multinomial-sum2} has at most $(N + 1)^k$ terms.

    For each $t \geq 0$ that is sufficiently small so that~\eqref{equ:min-mean-condition} holds if $c$ is replaced by $c + t$, let $H_{\max}(t)$ be the quantity corresponding to $H_{\max}$ if $c$ is replaced by $c + t$, and let
    \begin{equation*}
        M_-(t) :=  \max_{\sum_{i = 1}^k a_i N_i \,\leq\, (c+t) N_-} \frac{N_-!}{\prod_{i = 1}^k (N_i!)} .
    \end{equation*}
    Since $N_- \leq N \leq N_- + Ch$, we have that $c N_- \leq cN \leq (c + t)N_-$ for $t > 0$ and $N \geq Ch(c/t + 1)$.
    Hence, we get that
    \begin{equation}\label{equ:squeeze-maximum}
        M_-(0) \leq M \leq M_-(t) ,
    \end{equation}
    if $t > 0$ and $N \geq Ch(c/t + 1)$.
    Moreover, by Lemma~\ref{lem:multinomial} and Lemma~\ref{lem:max-entropy}, we get that, uniformly for bounded $t \geq 0$, it holds
    \begin{equation}\label{equ:M-minus-t}
        M_-(t) = \exp\!\left((H_{\max}(t) + o(1)) N_- \right) = \exp\!\left((H_{\max}(t) + o(1)) N\right) ,
    \end{equation}
    and $N \to +\infty$.
    Note that to deduce \eqref{equ:M-minus-t} we have to approximate $H(x_1, \dots, x_k)$, where $(x_1, \dots, x_k)$ is the point of maximum, with $H(N_1/N_-, \dots, N_k/N_-)$; and this can be done with an error of $O(1/N_-)$, since $H$ has bounded partial derivatives in a neighborhood of $(x_1, \dots, x_k)$.
    At this point, from~\eqref{equ:squeeze-maximum} and \eqref{equ:M-minus-t}, and noticing that $H_{\max}(t) \to H_{\max}$ as $t \to 0$, we get that
    \begin{equation}\label{equ:multinomial-sum3}
        M = \exp\!\left((H_{\max} + o(1)) N \right) ,
    \end{equation}
    as $N \to +\infty$.

    Finally, putting together \eqref{equ:multinomial-sum1}, \eqref{equ:multinomial-sum2}, and \eqref{equ:multinomial-sum3}, we get~\eqref{equ:multinomial-sum0}, as desired.
\end{proof}

\subsection{Further lemmas}

\begin{lemma}\label{lem:restricted-digits}
    Let $\mathcal{D} \subseteq \{1, \dots, b-1\}$ be nonempty, and let $(a_d)_{d \in \mathcal{D}}$ and $c$ be nonnegative real numbers such that
	\begin{equation*}
		\min_{d \,\in\, \mathcal{D}} a_d < c < \frac1{|\mathcal{D}|}\sum_{d \,\in\, \mathcal{D}} a_d .
	\end{equation*}
	Then the equation
	\begin{equation*}
		\sum_{d \,\in\, \mathcal{D}} (a_d - c) e^{a_d \lambda} = 0
	\end{equation*}
	has a unique solution $\lambda \in (-\infty, 0)$.
	Moreover, for each $C \geq 0$, we have that
    \begin{equation*}
		\sum \frac{N!}{N_1! \cdots N_{b-1}!} = \exp\!\big((H_{\max} + o(1)) N \big) ,
	\end{equation*}
	as $N \to +\infty$, where $N := \sum_{d=1}^{b-1} N_d$, the sum is over all integers $N_1, \dots, N_{b-1} \geq 0$ such that $\sum_{d \in \mathcal{D}} a_d N_d \leq c N$ and $N_d \leq C$ for each integer $d \in \{1,\dots,b-1\} \setminus \mathcal{D}$, 
	\begin{equation*}
		H_{\max} := -c \lambda + \log\!\left(\sum_{d \,\in\, \mathcal{D}} e^{a_d \lambda} \right) ,
	\end{equation*}
	and the $o(1)$ depends only on $(a_d)_{d \in \mathcal{D}}, c, h, C$.
\end{lemma}
\begin{proof}
	The claim is a direct consequence of Lemma~\ref{lem:max-entropy} and Lemma~\ref{lem:multinomial-sum}, after an appropriate change of the indices of the $N_i$'s.
\end{proof}

\begin{lemma}\label{lem:from-N-to-x}
    Let $\mathcal{S}$ be a set of positive integers, let $\mathcal{S}(x) := \mathcal{S} \cap [1, x]$ for every real number $x \geq 1$, let $\mathcal{S}_N := \mathcal{S} \cap {[b^{N-1}, b^N)}$ for each integer $N \geq 1$, and let $t \in (0, 1)$.
    If $|\mathcal{S}_N| \leq b^{(t + o(1))N}$ as $N \to +\infty$, then $|\mathcal{S}(x)| \leq x^{t + o(1)}$ as $x \to +\infty$.
    Moreover, the previous claim holds if the two inequalities are reversed, or replaced by equalities.
\end{lemma}
\begin{proof}
    For every real number $x \geq 1$, let $N_x := \lfloor \log x / \! \log b \rfloor + 1$ and $M_x := \tfrac1{2} t \lfloor \log x / \! \log b \rfloor$.
    On the one hand, if $|\mathcal{S}_N| \leq b^{(t + o(1))N}$ as $N \to +\infty$, then we get that
    \begin{equation*}
        |\mathcal{S}(x)| \leq x^{t/2} + \sum_{N \,=\, M_x + 1}^{N_x} |\mathcal{S}_N| \leq x^{t/2} + N_x b^{(t + o(1)) N_x} = x^{t + o(1)} ,
    \end{equation*}
    as $x \to +\infty$.
    On the other hand,  if $|\mathcal{S}_N| \geq b^{(t + o(1))N}$ as $N \to +\infty$, then we get that
    \begin{equation*}
        |\mathcal{S}(x)| \geq |\mathcal{S}_{N_x - 1}| \geq b^{(t + o(1))(N_x - 1)} = x^{t + o(1)} ,
    \end{equation*}
    as $x \to +\infty$.
\end{proof}

\begin{lemma}\label{lem:strange-bound}
    If $b \geq 6$ then
    \begin{equation}\label{equ:strange-bound1}
        \sum_{d \,\in\, \mathcal{D}} \log d > \frac{|\mathcal{D}|\log b}{2}
    \end{equation}
    for each $\mathcal{D} \in \Omega_b^*$.
\end{lemma}
\begin{proof}
    By Remark~\ref{rmk:Omega-b-star}, if $b$ is prime then $\mathcal{D} = \{1, \dots, b - 1\}$, and the claim follows easily.
    Hence, assume that $b$ is composite.
    Again by Remark~\ref{rmk:Omega-b-star}, we have that
    \begin{equation}\label{equ:D-description}
        \mathcal{D} = \big\{d \in \{1, \dots, b-1\} : p \nmid d \big\} ,
    \end{equation}
    where $p$ is a prime factor of $b$.
    With the aid of a computer, we can verify the claim for each composite $b \in [6, 84]$.
    Hence, assume that $b > 84$.
    Let $m := \lfloor (b-1)/p \rfloor$.
    From~\eqref{equ:D-description} it follows that $|\mathcal{D}| = b - 1 - m$ and
    \begin{equation*}
        \sum_{d \,\in\, \mathcal{D}} \log d = \log\!\big((b-1)!\big) - \sum_{k \,=\, 1}^m \log (pk) = \log\!\big((b-1)!\big) - m \log p - \log(m!) .
    \end{equation*}
    Therefore, we have that~\eqref{equ:strange-bound1} is equivalent to
    \begin{equation}\label{equ:strange-bound2}
        \log\!\big((b-1)!\big) - \tfrac1{2}(b-1)\log b > \log(m!) + m \log p - \tfrac1{2}m\log b .
    \end{equation}
    Employing the bounds $n \log n - n \leq \log(n!) \leq n \log n$, holding for every integer $n \geq 1$, we get that
    \begin{align*}
        \log\!\big((b-1)!\big) - \tfrac1{2}(b-1)\log b &\geq (b-1)\left(\tfrac1{2}\log(b-1) - 1 - \tfrac1{2}\log\!\big(b / (b-1)\big)\right) \\
        &> (b-1)\left(\tfrac1{2}\log(b-1) - 1.1\right)
    \end{align*}
    while
    \begin{align*}
        \log(m!) + m \log p - \tfrac1{2}m\log b &\leq m \log(mp) - \tfrac1{2} m \log (b - 1) \\
        &\leq \tfrac1{2} m \log (b - 1) \\
        &\leq \tfrac1{4} (b - 1) \log(b-1)
    \end{align*}
    since $mp \leq b-1$ and $m \leq (b - 1) / 2$.
    Hence, recalling that $b > 84$, we get that~\eqref{equ:strange-bound2} is satisfied.
\end{proof}

\begin{remark}
    One can check that \eqref{equ:strange-bound1} does not hold if $b \in \{3,4,5\}$ and $\mathcal{D} \in \Omega_b^*$.
\end{remark}

For all real numbers $s \geq 0$ and for every $\mathcal{D} \in \Omega_b^*$, define
\begin{equation*}
	\zeta_{\mathcal{D}}(s) := \sum_{d \,\in\, \mathcal{D}} \frac1{d^s} .
\end{equation*}
The next lemma regards an equation involving $\zeta_{\mathcal{D}}(s)$.

\begin{lemma}\label{lem:critical-equation}
	Let $\mathcal{D} \in \Omega_b^*$.
	If $v = 1$, or $v = 2$ and $b \geq 6$, then the equation
	\begin{equation}\label{equ:critical-equation}
        \left(s + \frac{v\log\!|\mathcal{D}|}{\log b}\right) \frac{\zeta_{\mathcal{D}}^\prime(s)}{\zeta_{\mathcal{D}}(s)} - \log\!\left(\frac{\zeta_{\mathcal{D}}(s)}{|\mathcal{D}|^v}\right) = 0 .
	\end{equation}
	has a unique solution $s \in (0, +\infty)$.
\end{lemma}
\begin{proof}
	Let $G_{b,\mathcal{D},v}(s)$ denote the left-hand side of \eqref{equ:critical-equation}.
    A quick computation yields that
    \begin{equation*}
        G_{b,\mathcal{D},v}^\prime(s) = \left(s + \frac{v\log\!|\mathcal{D}|}{\log b}\right) \frac{\zeta_{\mathcal{D}}(s) \, \zeta_{\mathcal{D}}^{\prime\prime}(s) - \big(\zeta_{\mathcal{D}}^\prime(s)\big)^2}{\big(\zeta_{\mathcal{D}}(s)\big)^2} > 0 ,
    \end{equation*}
    where we used the Cauchy--Schwarz inequality, which is strict since $|\mathcal{D}| \geq 2$.
    Moreover, we have that
    \begin{equation*}
        G_{b,\mathcal{D},v}(0) = -\left(1 + v\left(\frac{\sum_{d \in \mathcal{D}} \log d}{|\mathcal{D}|\log b} - 1\right)\right)\log\!|\mathcal{D}|,
    \end{equation*}
    and
    \begin{equation*}
        \lim_{s \,\to\, +\infty} G_{b,\mathcal{D},v}(s) = v \log\!|\mathcal{D}| \!\left(1 - \frac{\log \min(\mathcal{D})}{\log b}\right) > 0 .
    \end{equation*}
    Hence, if 
    \begin{equation}\label{equ:zero-test}
        1 + v\left(\frac{\sum_{d \in \mathcal{D}} \log d}{|\mathcal{D}|\log b} - 1\right) > 0
    \end{equation}
    then $G_{b,\mathcal{D},v}(s)$ has a unique zero in $(0, +\infty)$ (otherwise $G_{b,\mathcal{D},v}(s)$ has no zero).
    
    If $v = 1$ then it is clear that \eqref{equ:zero-test} holds.
    If $v = 2$ and $b \geq 6$, then~\eqref{equ:zero-test} is equivalent to~\eqref{equ:strange-bound1}, which in turn is true by Lemma~\ref{lem:strange-bound}.
\end{proof}

\section{Upper bound}\label{sec:upper-bound}

In light of Lemma~\ref{lem:critical-equation}, for every $\mathcal{D} \in \Omega_b^*$ let $s_{b,\mathcal{D},1}$ be the unique solution to \eqref{equ:critical-equation} with $v=1$ and define
\begin{equation*}
	z_{b,\mathcal{D}}^+ := \frac{\log|\mathcal{D}|}{\log b} \left(1 + \frac{\zeta_{\mathcal{D}}^\prime(s_{b,\mathcal{D},1})}{\zeta_{\mathcal{D}}(s_{b,\mathcal{D},1})\log b}\right) \quad\text{ and }\quad z_b^+ := \max_{\mathcal{D} \,\in\, \Omega_b^*} z_{b,\mathcal{D}}^+ .
\end{equation*}
It follows easily that $z_b^+ \in (0,1)$.

We have the following upper bound for the counting function of Zuckerman numbers.

\begin{theorem}\label{thm:upper-bound}
	We have that
	\begin{equation*}
		|\mathcal{Z}_b(x)| < x^{z_b^+ \,+\, o(1)} ,
	\end{equation*}
	as $x \to +\infty$.
\end{theorem}
\begin{proof}
    In light of Lemma~\ref{lem:Zb-as-union-of-ZbD} and Lemma~\ref{lem:from-N-to-x}, it suffices to prove that
	\begin{equation}\label{equ:upper-bound-ZbD}
		|\mathcal{Z}_{b,\mathcal{D},N}| < b^{(z_{b,\mathcal{D}}^+ \,+\, o(1))N} ,
	\end{equation}
	as $N \to +\infty$, for every $\mathcal{D} \in \Omega_b^*$.
	Pick an arbitrary $\mathcal{D} \in \Omega_b^*$.
    Let $s > 0$ be a constant to be defined later (depending only on $b$ and $\mathcal{D}$), and let
    \begin{equation}\label{equ:alpha-beta-gamma}
        \alpha := -\frac{\zeta_{\mathcal{D}}^\prime(s)}{ \zeta_{\mathcal{D}}(s) \log b},
        \quad \beta := \frac{\log\!|\mathcal{D}|}{\log b} (1 - \alpha) ,
        \quad\text{ and }\quad \gamma := \alpha s +     \frac{\log\!\big(\zeta_{\mathcal{D}}(s)\big)}{\log b}.
    \end{equation}
	For every integer $N \geq 1$, define the sets
	\begin{equation*}
		\mathcal{Z}_{b,\mathcal{D},N}^\prime := \big\{ n \in \mathcal{Z}_{b,\mathcal{D},N} : p_b(n) > b^{\alpha N} \big\}
	\end{equation*}
	and $\mathcal{Z}_{b,\mathcal{D},N}^{\prime\prime} := \mathcal{Z}_{b,\mathcal{D},N} \setminus \mathcal{Z}_{b,\mathcal{D},N}^\prime$.
	Thus $\mathcal{Z}_{b,\mathcal{D},N}^\prime$, $\mathcal{Z}_{b,\mathcal{D},N}^{\prime\prime}$ is a partition of $\mathcal{Z}_{b,\mathcal{D},N}$.

	Pick $n \in \mathcal{Z}_{b,\mathcal{D},N}$.
	For the sake of brevity, put $N_d := w_{b,d}(n)$ for each $d \in \{1, \dots, b - 1\}$.
	Hence, we have that $n$ has exactly $N$ digits, and $N_1 + \cdots + N_{b-1} = N$ (recall Lemma~\ref{lem:simple-facts}\ref{ite:simple-facts:1}).
	Moreover, recalling the definition of $\mathcal{Z}_{b,\mathcal{D}}$, we get that $N_d < W_b$ for every $d \in \mathcal{D}^c$, where $\mathcal{D}^c := \{1,\dots,b-1\} \setminus \mathcal{D}$.

	First, suppose that $n \in \mathcal{Z}_{b,\mathcal{D},N}^\prime$.
	Let $\ell$ be the unique integer such that $b^\ell \leq p_b(n) < b^{\ell + 1}$.
	Since $p_b(n) > b^{\alpha N}$, it follows that $\ell > \alpha N - 1$.
	Moreover, since $p_b(n)$ divides $n$, we have that $b^\ell \leq p_b(n) \leq n < b^N$, so that $\ell < N$.
	Recalling that $N_d < W_b$ for every $d \in \mathcal{D}^c$, we get that the number of possible choices for the $N - \ell$ most significant digits of $n$ is at most
	\begin{equation*}
		\sum_{j \,=\, 0}^{(W_b - 1)|\mathcal{D}^c|} \binom{N - \ell}{j} \, |\mathcal{D}^c|^j \, |\mathcal{D}|^{N - \ell - j} = N^{O(1)} |\mathcal{D}|^{N - \ell} < b^{(\beta + o(1)) N} ,
	\end{equation*}
	as $N \to +\infty$.
	Furthermore, since $b^\ell \leq p_b(n)$ and $p_b(n)$ divides $n$, we get that for each of the previous choices there is at most one choice for the $\ell$ least significant digits of $n$ (a similar idea is used in the proof of \cite[Lemma~2]{MR2046944}).
	Therefore, we obtain that
	\begin{equation}\label{equ:upper-bound-ZbD1}
		|\mathcal{Z}_{b,\mathcal{D},N}^\prime| < b^{(\beta + o(1)) N} ,
	\end{equation}
	as $N \to +\infty$.

	Suppose that $n \in \mathcal{Z}_{b,\mathcal{D},N}^{\prime\prime}$.
	Since $p_b(n) \leq b^{\alpha N}$, we have that $\sum_{d \in \mathcal{D}} (\log d) N_d \leq (\alpha \log b) N$.
	Moreover, by elementary combinatorics, for fixed values of $N_1, \dots, N_{b-1}$ there are at most
	\begin{equation*}
		\frac{N!}{N_1! \cdots N_{b-1}!}
	\end{equation*}
	possible values of $n$.
	Therefore, we can apply Lemma~\ref{lem:restricted-digits} with $a_d = \log d$ for $d \in \mathcal{D}$, $c = \alpha \log b$, $C = W_b - 1$, and $\lambda = -s$ (as a consequence of the definition of $\alpha$, and also recalling Remark~\ref{rmk:converse}).
	This yields
	\begin{equation}\label{equ:upper-bound-ZbD2}
		|\mathcal{Z}_{b,\mathcal{D},N}^{\prime\prime}| < b^{(\gamma + o(1)) N} ,
	\end{equation}
	as $N \to +\infty$.
    
    At this point, by choosing $s = s_{b,\mathcal{D},1}$, from~\eqref{equ:critical-equation} we get that $\beta = \gamma = z_{b,\mathcal{D}}^+$.
	Therefore, putting together \eqref{equ:upper-bound-ZbD1} and \eqref{equ:upper-bound-ZbD2}, we obtain~\eqref{equ:upper-bound-ZbD}, as desired.
\end{proof}

\section{Lower bound for base 10}\label{sec:lower-bound-10}

We have the following lower bound for the counting function of base-$10$ Zuckerman numbers.

\begin{theorem}\label{thm:lower-bound-10}
    We have that
    \begin{equation*}
        |\mathcal{Z}_{10}(x)| > x^{0.204}
    \end{equation*}
    as $x \to +\infty$.
\end{theorem}
\begin{proof}
    In light of Lemma~\ref{lem:from-N-to-x}, it suffices to prove that
    \begin{equation*}
        |\mathcal{Z}_{10,N}| > 10^{0.204 N} ,
    \end{equation*}
    as $N \to +\infty$.
    For each integer $n \geq 0$, let $\nu(n)$ be the maximum integer $v$ such that $2^v$ divides $p_{10}(n)$ (where $p_{10}(0) := 1$), and let $\mathcal{A}_n$ be the set of $n$-digit numbers whose digits belong to $\{1, 2, 4, 8\}$.
    Let $\ell$ be a positive integer to be determined later, and put
    \begin{equation}\label{equ:min-max-def}
        \delta := \frac1{\ell} \max_{0 \,\leq\, x \,<\, 2^\ell} \min_{\substack{y \,\in\, \mathcal{A}_\ell \\ y \,\equiv\, x \!\!\!\pmod {2^{\ell}}}} \nu(y) ,
    \end{equation}
    where $\min \varnothing := +\infty$.
    Assume that $\delta < 1$.
    Let us prove that for each integer $k \geq 0$ there exists $n_k \in \mathcal{A}_{k\ell}$ such that $2^{k\ell}$ divides $n_k$ and $\nu(n_k) \leq \delta k \ell$.
    We proceed by induction.
    For $k = 0$ it suffices to put $n_0 := 0$.
    Suppose that we proved the existence of $n_k$ for some integer $k \geq 0$.
    Let us construct $n_{k+1}$.
    By the induction hypothesis, we have that $m_k := n_k / 2^{k\ell}$ is an integer.
    Moreover, by definition~\eqref{equ:min-max-def}, there exists $y_k \in \mathcal{A}_\ell$ such that $y_k \equiv -5^{-k\ell} m_k \pmod {2^\ell}$ and $\nu(y_k) \leq \delta \ell$.
    Put $n_{k + 1} := 10^{k\ell} y_k + n_k$.
    It is clear that $n_{k + 1} \in \mathcal{A}_{(k+1)\ell}$.
    Furthermore, we have that $n_{k + 1} \equiv 2^{k\ell} (5^{k\ell} y_k + m_k) \equiv 0 \pmod {2^{(k+1)\ell}}$ and $\nu(n_{k+1}) = \nu(y_k) + \nu(n_k) \leq \delta \ell + \delta k\ell = \delta(k+1)\ell$, as desired.
    The proof of the claim is complete.

    Let $N$ be a sufficiently large integer, and put $k := \lceil \alpha N / \ell \rceil$, $N_\low := k\ell$, and $N_\high := N - N_\low$, where $\alpha \in (0, 1)$ is a constant to be determined later.
    Suppose that $n = 10^{N_\low} n_\high + n_k$, where $n_\high \in \mathcal{A}_{N_\high}$ satisfies
    \begin{equation*}
        \nu(n_\high) \leq \frac{\alpha(1 - \delta)}{1 - \alpha} N_{\high} .
    \end{equation*}
    Then we have that $n \in \mathcal{A}_N$ and
    \begin{equation*}
        \nu(n) = \nu(n_\high) + \nu(n_k) \leq \frac{\alpha(1 - \delta)}{1 - \alpha} N_{\high} + \delta N_\low \leq (1 - \delta) N_\low + \delta N_\low = N_\low .
    \end{equation*}
    Consequently, recalling that $2^{N_\low}$ divides $n_k$, and since $p_{10}(n) = 2^{\nu(n)}$, we get that $p_{10}(n)$ divides $n$, so that $n \in \mathcal{Z}_{10,N}$.

    At this point, setting $\mathcal{D} = \{1, 2, 4, 8\}$, $a_1 = 0$, $a_2 = 1$, $a_4 = 2$, $a_8 = 3$, $c = \alpha(1 - \delta)/(1 - \alpha)$, and $C = 0$, the number of possible choices for $n_\high$ can be estimated by Lemma~\ref{lem:restricted-digits} (if the hypothesis on $a_1,a_2,a_4,a_8,c$ are satified), and it is equal to, say,
    \begin{equation*}
        \exp\!\big((H_{\alpha,\delta} + o(1)) N_\high \big) = \exp\!\big(((1 - \alpha) H_{\alpha,\delta} + o(1)) N \big) = 10^{((1 - \alpha) H_{\alpha,\delta} /\! \log{10} + o(1))N}.
    \end{equation*}
    Finally, we run the numbers.
    Picking $\ell = 24$, a computation~\cite{Repo} (see Remark~\ref{rmk:delta-computation}) shows that $\delta = 1/2$.
    Then the choice $\alpha = 1/2$ yields that $H_{\alpha,\delta} > 0.940$ and
    \begin{equation*}
        \frac{(1 - \alpha) H_{\alpha,\delta}}{\log 10} > 0.204 ,
    \end{equation*}
    as desired.
\end{proof}

\begin{remark}
    The proof of Theorem~\ref{thm:lower-bound-10} is similar to the proof of the lower bound of \cite[Theorem~1(i)]{MR2298113}, but with a significant difference. 
    Instead of the max-min approach employing~\eqref{equ:min-max-def}, in~\cite{MR2298113} the sequence $(n_k)_{k \geq 0}$ is constructed by an inductive process that allows one to make some choices. 
    Then it is shown that the choices leading to the worst-case scenario do not happen too often. 
    This provided a worse parameter $\delta = 3/4$. 
    Also, as a minor difference, we constructed $(n_k)_{k \geq 0}$ by employing the digits $1,2,4,8$, while~\cite{MR2298113} only uses the digits $1,2,4$.
\end{remark}

\begin{remark}\label{rmk:delta-computation}
    We have that $24$ is the smallest value of $\ell$ for which $\delta \leq 1/2$, and we did not find a value of $\ell$ for which $\delta < 1/2$ (we tried up to $\ell = 28$).
    Note that the search space to compute $\delta$ has $|\mathcal{A}_\ell| = 4^{\ell}$ elements, which for $\ell = 24$ is already impractical for a brute-force search.
    However, employing some basic pruning techniques allowed us to complete the computation within a few seconds.
\end{remark}

\begin{remark}
    The strategy of the proof of Theorem~\ref{thm:lower-bound-10} might generalize to other composite values of $b$.
    The idea would be to fix a prime factor $q$ of $b$, let $\mathcal{A}_n$ be the set of $n$-digit numbers whose digits are all powers of $q$, let $\nu(n)$ be the maximum integer $v$ such that $q^v$ divides $p_b(n)$, and let $\delta$ be defined as in~\eqref{equ:min-max-def}, but with $2^\ell$ replaced by $q^\ell$. 
    If $\delta < 1$ for a sufficiently large $\ell$, then one gets a lower bound for $|\mathcal{Z}_b(x)|$.
\end{remark}

\section{Heuristic}\label{sec:heuristic}

For every $\mathcal{D} \in \Omega_b^*$, let
\begin{equation*}
	z_{b,\mathcal{D}} := \frac{\log\!\big(\zeta_{\mathcal{D}}(1)\big)}{\log b} , \quad\text{ and let }\quad z_b := \max_{\mathcal{D} \,\in\, \Omega_b^*} z_{b,\mathcal{D}} .
\end{equation*}
In fact, in light of Remark~\ref{rmk:Omega-b-star}, it follows easily that
\begin{equation*}
	z_b = \frac1{\log b} \log\!\Big(\sum_{1 \,\leq\, d \,<\, b, \; P^+(b) \,\nmid\, d} \frac1{d} \Big) ,
\end{equation*}
where $P^+(b)$ denotes the greatest prime factor of $b$.

In this section, we describe a heuristic suggesting that, for each $\mathcal{D} \in \Omega_b^*$, we have that
\begin{equation}\label{equ:heuristic-ZbDN}
	|\mathcal{Z}_{b,\mathcal{D},N}| = b^{(z_{b,\mathcal{D}} + o(1))N} ,
\end{equation}
as $N \to +\infty$, and consequently, by Lemma~\ref{lem:Zb-as-union-of-ZbD} and Lemma~\ref{lem:from-N-to-x}, that
\begin{equation*}
	|\mathcal{Z}_{b}(x)| = x^{z_b + o(1)} ,
\end{equation*}
as $x \to +\infty$.

The heuristic suggesting \eqref{equ:heuristic-ZbDN} works as follows.
Fix $\mathcal{D} \in \Omega_b^*$, let $N \geq 1$ be an integer, and let $n$ be a positive integer whose representation has exactly $N$ digits.
In light of Lemma~\ref{lem:simple-facts}\ref{ite:simple-facts:1}, suppose that $0 \notin \mathcal{D}_b(n)$, and put $N_d := w_{b,d}(n)$ for each $d \in \{1, \dots, b - 1\}$.
Hence, we have that $N = \sum_{d = 1}^{b-1} N_d$.
For fixed values of $N_1, \dots, N_{b-1}$, the number of possible values of $n$ is equal to
\begin{equation*}
	\frac{N!}{\prod_{d=1}^{b-1} (N_d!)} .
\end{equation*}
Each of these $n$'s is a Zuckerman number if and only if it is divisible by $p_b(n) = \prod_{d=1}^{b-1} d^{N_d}$.
Heuristically, the probability that this occurs is equal to $1/p_b(n)$.
Hence, the expected number of Zuckerman numbers among the aforementioned $n$'s is equal to
\begin{equation}\label{equ:heuristic-addend}
	\frac{N!}{\prod_{d=1}^{b-1} \big((N_d!) d^{N_d}\big)} .
\end{equation}
At this point, to (heuristically) compute $|\mathcal{Z}_{b,\mathcal{D},N}|$, we sum~\eqref{equ:heuristic-addend} over all the values of $N_1, \dots, N_{b-1}$ with $N_d < W_b$ for each $d \in \mathcal{D}^c := \{1, \dots, b - 1\} \setminus \mathcal{D}$.
Letting $N_{\mathcal{D}} := \sum_{d \in \mathcal{D}} N_d$, this yields that
\begin{align*}
	|\mathcal{Z}_{b,\mathcal{D},N}| &= \sum \frac{N!}{\prod_{d=1}^{b-1} \big((N_d!) d^{N_d}\big)} \\
	&= \sum_{N_{\mathcal{D}} \,=\, 1}^N \; \sum_{\substack{\sum_{d \in \mathcal{D}} N_d \,=\, N_{\mathcal{D}}}} \frac{N_{\mathcal{D}}!}{\prod_{d \in \mathcal{D}} \big((N_d!) d^{N_d}\big)}
	\cdot \sum_{\substack{N_d \,<\, W_b \\[1pt] d \,\in\, \mathcal{D}^c}} \frac{\prod_{k=1}^{N_{\mathcal{D}^c}} (N_{\mathcal{D}} + k)}{\prod_{d \in \mathcal{D}^c} \big((N_d!) d^{N_d}\big)} \\
	&= \sum_{N_{\mathcal{D}} \,=\, 1}^N \big(\zeta_{\mathcal{D}}(1)\big)^{N_{\mathcal{D}}}
	\cdot b^{o(1) N} = b^{(z_{b,\mathcal{D}} + o(1)) N} ,
\end{align*}
as $N \to +\infty$, where in the last equality we employed the multinomial theorem.
Thus \eqref{equ:heuristic-ZbDN} is proved.

\section{Algorithms}\label{sec:algorithms}

In this section, we describe some algorithms to count, or to enumerate, the Zuckerman numbers that have exactly $N$ digits.
The first three algorithms (Sections~\ref{sec:brute-force}, \ref{sec:enum-multiples}, and \ref{sec:meet-in-the-middle}) are asymptotically worse than the last two (Sections~\ref{sec:improved-algorithm} and~\ref{sec:algorithm-10}).
Hence, we only give a brief description of them, without diving into the details, and we omit possible improvements coming from Lemma~\ref{lem:Zb-as-union-of-ZbD}.

\subsection{Brute force}\label{sec:brute-force}

Of course, the simplest algorithm proceeds by brute force.
Each integer with $N$ digits, which are all nonzero (Lemma~\ref{lem:simple-facts}\ref{ite:simple-facts:1}), is tested to determine if it is a Zuckerman number or not.
The complexity of this algorithm is of the order of $(b-1)^N$.

\subsection{Enumerating multiples}\label{sec:enum-multiples}

First, the algorithm runs over the possible values $P$ of the product of $N$ digits such that $P > b^{\alpha N}$, where $\alpha > 0$ is a constant.
For each $P$, the algorithm runs over the multiples of $P$ that have $N$ digits, checking if each of them has digits consistent with the product $P$, and thus is a Zuckerman number.
Second, the algorithms runs over the $N$-digit numbers $n$ such that $p_b(n) \leq b^{\alpha N}$, and determines which of them are Zuckerman numbers.

The complexity of the first part if $b^{(1 - \alpha + o(1)) N}$, while (with reasonings similar to those leading to~\eqref{equ:upper-bound-ZbD2}), the complexity of the second part is $b^{(\delta + o(1))N}$, where
\begin{equation*}
    \delta := \alpha s + \frac{\log\!\big(\zeta_b(s)\big)}{\log b} , \quad\text{ with }\quad \zeta_b(s) := \zeta_{\{1,\dots,b-1\}}(s) ,
\end{equation*}
and $s > 0$ is arbitrary.
The optimal choice for $\alpha$ and $s$ is taking $s$ to be the unique solution of the equation
\begin{equation*}
    \frac{(s + 1)\zeta_b^\prime(s)}{\zeta_b(s)}- \log\!\left(\frac{\zeta_b(s)}{b}\right) = 0
\end{equation*}
(see \cite[Theorem~2.2]{MR4181552}) and setting
\begin{equation*}
    \alpha = \frac{\log\!\big(b / \zeta_b(s)\big)}{(s + 1)\log b} .
\end{equation*}
In particular, for $b = 10$ the algorithm has complexity of the order of $10^{0.717 N}$.

This algorithm can also be easily modified to enumerate Zuckerman numbers, which increases the complexity by an additive term $|\mathcal{Z}_{b,N}|$.

\subsubsection{Dynamic programming}

It is also possible to count the Zuckerman numbers $n$ such that $p_b(n) \leq b^{\alpha N}$ by using a dynamic programming approach.
Let $f[p][i][\mathcal{S}]$ denote the number of $n$ such that $n \equiv i \pmod p$ and the multiset of digits of $n$ is equal to $\mathcal{S}$.
Then $f[p][\,\cdots][\mathcal{T}]$, with $|\mathcal{T}| = k + 1$, can be computed from $f[p][\,\cdots][\mathcal{S}]$, with $|\mathcal{S}| = k$, by initializing it with zeros and then adding $f[p][i][\mathcal{S}]$ to
\begin{equation*}
    f[p][(ib+d)\bmod p][\mathcal{S}\cup \{d\}]
\end{equation*}
for all the possible values of $i$, $d$, and $\mathcal{S}$. 
(This corresponds to adding a new least significant digit $d$.)
Finally, we only need to compute
\begin{equation}\label{equ:dynamic-sum}
    \sum_{p \,\leq\, b^{\alpha N}, \, |\mathcal{S}| \,=\, N, \, \prod_{v\in \mathcal{S}} v \,=\, p} f[p][0][\mathcal{S}] .
\end{equation}
Since the number of $\mathcal{S}$'s that we have to consider is negligible (of the order of $b^{o(N)}$), the complexity of computing~\eqref{equ:dynamic-sum} is of the order of $b^{(\alpha + o(1)) N}$.
By setting $\alpha=1/2$, we get an overall complexity of $b^{(1/2 + o(1))N}$.

\subsection{Meet in the middle}\label{sec:meet-in-the-middle}

This algorithm follows a meet-in-the-middle approach.
Letting $N_\low := \lfloor N / 2 \rfloor$ and $N_\high := N - N_\low$, each $N$-digit number $n$ is written as $n = n_\high \, b^{N_\low} + n_\low$, where $n_\high$ and $n_\low$ are $N_\high$-digit and $N_\low$-digit numbers, respectively.
Hence, we have that $n$ is a Zuckerman number if and only if $n_\low \equiv -n_\high \, b^{N_\low} \pmod P$, where $P := P_\high P_\low$, $P_\high := p_b(n_\high)$, and $P_\low := p_b(n_\low)$.

The algorithm runs over the possible values $P_\high$ and $P_\low$ of the product of $N_\high$ digits and $N_\low$ digits, respectively, and for each pair $(P_\high, P_\low)$ does the following computation.
First, it builds a table $T[r]$, with $r=0,\dots,P-1$, such that $T[r]$ is equal to the number of $n_\high$ satisfying $p_b(n_\high) = P_\high$ and $r \equiv -n_\high \, b^{N_\low} \pmod P$.
Second, for each $n_\low$ with $p_b(n_\low) = P_\low$ it increases the counter of Zuckerman numbers by $T[n_\low \bmod P]$.

The complexity of this algorithm is of the order of $b^{(1/2 + o(1))N}$.
Also, by storing $n_\high$ in the table $T$, the algorithm can be easily modified to enumerate Zuckerman numbers.
This increases the complexity to $b^{(1/2 + o(1))N} + |\mathcal{Z}_{b,N}|$.

\subsection{An improved algorithm}\label{sec:improved-algorithm}

The strategy of this algorithm, which we call $\textsf{ZuckermanCount}$, is the following.
Let $n$ be an integer with $N$ digits, and let $N_d := w_{b,d}(n)$ for $d=0,\dots,b - 1$.
We assume that $N_0 = 0$, by Lemma~\ref{lem:simple-facts}\ref{ite:simple-facts:1}.
The first part of the algorithm (Figure~\ref{fig:zuckerman-count}) runs over the possible values of $N_0, \dots, N_{b-1}$, taking into account the restriction given by Lemma~\ref{lem:Zb-as-union-of-ZbD}, and computes the product of digits $P := \prod_{d = 1}^{b - 1} d^{N_d}$.
Then, depending if $P > b^{\alpha N}$ or not, where $\alpha > 0$ is a constant defined in the proof of Theorem~\ref{thm:counting-complexity}, the subroutine $\textsf{LargeProduct}$ (Figure~\ref{fig:large-product}) or $\textsf{SmallProduct}$ (Figure~\ref{fig:small-product}) is called, respectively.

The subroutine $\textsf{LargeProduct}$ counts the number of $n$'s that are divisible by $P$ by following the same strategy of the first part of the proof of Theorem~\ref{thm:upper-bound}.
More precisely, $\textsf{LargeProduct}$ runs over the possible values of the most significant $N - \ell$ digits of $n$, where $\ell := \lfloor \log P /\! \log b \rfloor$, and uniquely determines the remaining $\ell$ digits of $n$ by the condition that $P$ divides $n$.

The subroutine $\textsf{SmallProduct}$ counts the number of $n$'s that are divisible by $P$ by using a meet-in-the-middle approach similar to that of Section~\ref{sec:meet-in-the-middle}.
The idea is the following.
Let $n_\high$ and $n_\low$ be the unique integers such that $n = n_\high \, b^{N_\low} + n_\low$, $1 \leq n_\high < b^{N_\high}$, and $1 \leq n_\low < b^{N_\low}$, where $N_\high$ is an integer to be defined later and $N_\low := N - N_\high$.
Then $P$ divides $n$ if and only if $n_\low \equiv -n_\high \, b^{N_\low} \pmod P$.
Hence, $\textsf{SmallProduct}$ first builds a table $T[r]$, with $r=0,\dots,P-1$, such that $T[r]$ is equal to the number of $n_\high$'s satisfying $r \equiv -n_\high \, b^{N_\low} \pmod P$.
Then, $\textsf{SmallProduct}$ runs over the possible values of $n_\low$, and increases the counter of Zuckerman numbers by $T[n_\low \bmod P]$.

The choice of $N_\high$ is made so that the number of values of $n_\low$ and the number of values of $n_\high$ are both of the order of $M^{1/2}$, where $M := N! / (N_0! \cdots N_{b-1}!)$.
This requires some considerations on what we call the \emph{separating index} $\sigma_b(n)$.
For each integer $n$, whose digits are all nonzero, let $\pi_b(n)$ be the number of integers that can be obtained by permuting the digits of $n$.
For every positive integer $i \leq N$, let $n[\colon \!i]$ denote the $i$-digit integer consisting of the $i$ most significant digits of $n$.
We define $\sigma_b(n)$ as the minimum positive integer $i$ such that $\pi_b\!\big(n[\colon \!i]\big) \geq M^{1/2}$.
It follows easily that $\sigma_b(n)$ is well defined and satisfies $M^{1/2} \leq \pi_b\!\big(n[\colon \!\sigma_b(n)]\big) < N M^{1/2}$, since adding one digit increases the number of permutations by a factor at most equal to $N$.
The algorithm constructs $N_\high$ so that, a posteriori, we have that $N_\high = \sigma_b(n)$.
This requires to discard the integers $n$ such that $\pi_b\!\big(n[\colon \! (N_\high - 1)]\big) \geq M^{1/2}$.
In this way, the number of possible values for $n_\high$ and $n_\low$ are not exceeding $NM^{1/2}$ and $M^{1/2}$, respectively.

By storing $n_\high$ in the table $T$, the algorithm can be easily modified to enumerate Zuckerman numbers.
This increases the complexity by an additive term $|\mathcal{Z}_{b,N}|$.

\subsubsection{Complexity analysis}

For every $\mathcal{D} \in \Omega_b^*$, if $b \in \{3,4,5\}$ then let $s_{b,\mathcal{D},2} := 0$, while if $b \geq 6$ then, in light of Lemma~\ref{lem:critical-equation}, let $s_{b,\mathcal{D},2}$ be the unique solution to \eqref{equ:critical-equation} with $v=2$.
Then define
\begin{equation*}
	z_{b,\mathcal{D}}^* := \frac{\log|\mathcal{D}|}{\log b} \left(1 + \frac{\zeta_{\mathcal{D}}^\prime(s_{b,\mathcal{D},2})}{\zeta_{\mathcal{D}}(s_{b,\mathcal{D},2})\log b}\right) \quad\text{ and }\quad z_b^* := \max_{\mathcal{D} \,\in\, \Omega_b^*} z_{b,\mathcal{D}}^* .
\end{equation*}
It follows easily that $z_b^* \in (0,1)$.

\begin{theorem}\label{thm:counting-complexity}
    The algorithm~$\textsf{ZuckermanCount}$ (Figure~\ref{fig:zuckerman-count}) has complexity $b^{(z_b^* + o(1))N}$, as $N \to +\infty$.
\end{theorem}
\begin{proof}
    First, note that the loop in \textsf{ZuckermanCount} has at most $N^b = b^{o(N)}$ steps, which is negligible for our estimate.
    Hence, it suffices to estimate the complexities of $\textsf{LargeProduct}$ and $\textsf{SmallProduct}$.
    Furthermore, in light of Lemma~\ref{lem:Zb-as-union-of-ZbD}, we can compute the complexities of $\textsf{LargeProduct}$ and $\textsf{SmallProduct}$ when they count the Zuckerman numbers in $\mathcal{Z}_{b,\mathcal{D},N}$, for a fixed $\mathcal{D} \in \Omega_b^*$, and then consider the worst-case $\mathcal{D}$.
    
    Let $s > 0$ be a constant to be defined later (depending only on $b$ and $\mathcal{D}$), and let $\alpha,\beta,\gamma$ be defined as in~\eqref{equ:alpha-beta-gamma}.
 
    For $\textsf{LargeProduct}$, the complexity can be estimated exactly as the upper bound~\eqref{equ:upper-bound-ZbD1}, hence it is at most $b^{(\beta + o(1))N}$.
    
    For $\textsf{SmallProduct}$, the number of steps in the outer loop is negligible, and consequently the complexity is at most the cost of generating all the $n_\high$'s and $n_\low$'s, which by construction is at most $N M^{1/2}$.
    Thus, ignoring negligible factors, the complexity of $\textsf{SmallProduct}$ is at most
    \begin{align*}
        \sum_{\prod_{d \in \mathcal{D}} d^{N_d} \,\leq\, b^{\alpha N}}& \left(\frac{N!}{N_1! \cdots N_{b-1}!} \right)^{1/2}
        \leq \left(\sum_{\prod_{d \in \mathcal{D}} d^{N_d} \,\leq\, b^{\alpha N}} 1 \right)^{1/2} \!\! \left(\sum_{\prod_{d \in \mathcal{D}} d^{N_d} \,\leq\, b^{\alpha N}} \frac{N!}{N_1!\cdots N_{b-1}!} \right)^{1/2} \\
        &\leq N^{b/2} \!\! \left(\sum_{\prod_{d \in \mathcal{D}} d^{N_d} \,\leq\, b^{\alpha N}} \frac{N!}{N_1!\cdots N_{b-1}!} \right)^{1/2} < b^{(\gamma/2 + o(1))N} ,
    \end{align*}
    where we employed the Cauchy--Schwarz inequality and Lemma~\ref{lem:restricted-digits} with $a_d = \log d$ for $d \in \mathcal{D}$, $c = \alpha \log b$, $C = W_b - 1$, and $\lambda = -s$ (recalling also Remark~\ref{rmk:converse}).
    
    It remains to choose $s$ so that $\max(\beta, \gamma/2)$ is minimal.
    If $b \geq 6$ then, in light of Lemma~\ref{lem:critical-equation}, we choose $s = s_{b,\mathcal{D},2}$, so that $\beta = \gamma / 2 = z_{b,\mathcal{D}}^*$.
    Hence, we obtain that the complexities of $\textsf{LargeProduct}$ and $\textsf{SmallProduct}$ are both equal to $b^{(z_{b,\mathcal{D}}^* + o(1))N}$.
    Finally, considering the worst-case scenario for $\mathcal{D}$, we get the complexity $\textsf{ZuckermanCount}$ is equal to $b^{(z_b^* + o(1))N}$.
    If $b \in \{3, 4, 5\}$, then by Remark~\ref{rmk:Omega-b-star} we have that $|\Omega_b| = 1$.
    Moreover, it can be verified that $\max(\beta, \gamma/2) = \beta$ and that the optimal choice is taking $s$ arbitrary small, as it is done in the definition of $z_b^*$ for $b \in \{3,4,5\}$.
\end{proof}

\subsection{A further improved algorithm for base 10}\label{sec:algorithm-10}

We now describe an ad hoc algorithm for $b = 10$, which we call $\textsf{ZuckermanCount10}$.
It is defined as $\textsf{ZuckermanCount}$ (with $b=10$) but the subroutine $\textsf{SmallProduct}$ is replaced by $\textsf{SmallProduct10}$, see Figure~\ref{fig:small-product-10}.

The idea on which $\textsf{SmallProduct10}$ is based is the following.
Let $N_\low$, $N_\mi$, and $N_\high$ be nonnegative integers such that $N_\low + N_\mi + N_\high = N$.
We write each $N$-digit integer as
\begin{equation}\label{equ:high-middle-low}
    n = 10^{N_\mi + N_\low} n_\high + 10^{N_\low} n_\mi + n_\low ,
\end{equation}
where $n_\high$, $n_\mi$, and $n_\low$ are $N_\high$-digit, $N_\mi$-digit, and $N_\low$-digit integers, respectively.
Let $v$ be the greatest integer such that $2^v$ divides $p_{10}(n)$.
Suppose that
\begin{equation}\label{equ:N-lo-bound}
    N_\low \leq \min\!\left\{ \frac{\log 2}{\log 10} \, v, \frac{\log 2}{\log 5} \, N_\mi \right\} .
\end{equation}
Let $u := \lceil N_\low \log 10 /\! \log 2 \rceil$, so that $n_\low < 10^{N_\low} \leq 2^u$.
From~\eqref{equ:N-lo-bound}, we get that $N_\low \leq v \log 2 /\! \log 10$ and so $u \leq v$.
Moreover, from~\eqref{equ:N-lo-bound} we have that $N_\low \leq N_\mi \log 2 /\!\log 5$, which implies that $N_\mi + N_\low \geq N_\low \log 10 /\! \log 2$, and so $N_\mi + N_\low \geq u$.

If $n$ is a Zuckerman number, then $2^v$ divides $n$.
Hence, by the previous considerations, we obtain that $2^u$ divides both $n$ and $10^{N_\mi + N_\low}$.
Thus we have that $n_\low \equiv -10^{N_\low} n_\mi \pmod {2^u}$ and $0 \leq n_\low < 2^u$.
Consequently, we get that $n_\low$ is uniquely determined by $n_\mi$, $N_\low$, and $v$.

At this point, the idea is to follow a meet-in-the-middle approach on the possible permutations of the digits of $n_\high$ and $n_\mi$, similarly to how it is done in $\textsf{SmallProduct}$, but with the advantage that we are considering permutations of only $N_\high + N_\mi$ digits, instead of $N$ digits.

More precisely, we put $\alpha := 0.57992$, $\gamma := 0.37938$, and
\begin{equation*}
    N_\low := \left\lfloor \min\!\left\{\frac{\log 2}{\log 10} \, v, \frac{\log 10 \log 2}{\log 8 \log 5} \, \gamma N\right\}\right\rfloor .
\end{equation*}
Then we compute $N_\high^*$ and $N_\mi^*$ to perform a balanced meet-in-the-middle for the permutations of the most significant $N - N_\low$ digits of $n$, and we set
\begin{equation*}
    N_\mi := \left\lfloor \max\!\left\{ \frac{\log 10}{\log 8} \,\gamma N , N_\mi^* \right\} \right\rfloor ,
\end{equation*}
which ensures~\eqref{equ:N-lo-bound}, and $N_\high := N - N_\low - N_\mi$.

\begin{theorem}\label{thm:counting-complexity-10}
    The algorithm~$\textsf{ZuckermanCount10}$ has complexity at most $10^{\gamma N}$, as $N \to +\infty$.
\end{theorem}
\begin{proof}
    The proof is similar to that of Theorem~\ref{thm:counting-complexity}, thus we only sketch it.
    First, the number of steps of the outer loops is negligible, and the complexity of $\textsf{LargeProduct}$ is $10^{(1 - \alpha +o(1)) N}$.
    Hence, it remains only to estimate the complexity of $\textsf{SmallProduct10}$.
    
    Note that $\Omega_{10}^* = \{\mathcal{D}_2, \mathcal{D}_5\}$, where $\mathcal{D}_d := \{1,2,\dots,9\} \setminus \{d\}$ for $d \in \{2, 5\}$.
    We consider only when $\textsf{SmallProduct10}$ counts the Zuckerman numbers in $\mathcal{Z}_{10,\mathcal{D}_5,N}$.
    For $\mathcal{Z}_{10,\mathcal{D}_2,N}$ one applies a similar reasoning, and can verify that $\mathcal{D}_5$ is indeed the worst-case scenario.

    If $N_\mi = N_\mi^*$ then, with a reasoning similar to that of the proof of Theorem~\ref{thm:counting-complexity}, the complexity of $\textsf{SmallProduct10}$ is of the order of
    \begin{equation*}
        \left(\sum \frac{(N - N_\low)!}{N_1! \cdots N_9!}\right)^{1/2}
    \end{equation*}
    where the sum is over all integers $N_1, \dots, N_9 \geq 0$ such that $\sum_{d=1}^9 N_d = N - N_\low$, $N_5 = 0$, $\sum_{d=1}^9 (\log d) N_d \leq \alpha \log 10$, while
    \begin{equation*}
        N_\low := \left\lfloor \min\!\left\{\frac{\log 2}{\log 10} (N_2 + 2N_4 + N_6 + 3N_8), \frac{\log 10 \log 2}{\log 8 \log 5} \, \gamma N\right\}\right\rfloor .
    \end{equation*}
    Therefore, using Lemma~\ref{lem:multinomial}, and a strategy similar to that of the proof of Lemma~\ref{lem:restricted-digits}, we get that $\textsf{SmallProduct10}$ has complexity $10^{\gamma N}$, since $\gamma$ is (slightly) greater than the maximum of
    \begin{equation}\label{equ:involved-max}
        -\frac1{2\log 10} \, (1 - y) \sum_{d \,=\, 1}^9 \frac{x_d}{1 - y} \log\!\left(\frac{x_d}{1 - y}\right) ,
    \end{equation}
    where 
    \begin{equation*}
        y := \min\!\left\{\frac{\log 2}{\log 10} (x_2 + 2x_4 + x_6 + 3x_8), \frac{\log 10 \log 2}{\log 8 \log 5} \gamma \right\} ,
    \end{equation*}
    and $x_1, \dots, x_9 \geq 0$ satisfy the constraints $\sum_{d=1}^9 x_d = 1 - y$, $x_5 = 0$, and $\sum_{d=1}^9 (\log d) x_d \leq \alpha \log 10$.
    (Essentially, $x_d = N_d / N$ and $y = N_\low / N$.)
    See Remark~\ref{rmk:involved-max} for more details about the computation of such maximum.
    
    If $N_\mi > N_\mi^*$ (and so $N_\high < N_\mi^*$) then the meet-in-the-middle is unbalanced, with the complexity of computing the permutations of the digits of $n_\mi$ dominating the overall computation.
    However, since $N_\mi \leq  N \gamma \log 10 /\! \log 8$, we have that the number of permutations of the digits of $n_\mi$ is at most $8^{N_\mi} \leq 10^{\gamma N}$.
    Hence, in any case, the complexity of $\textsf{SmallProduct10}$ is at most $10^{(\gamma + o(1)) N}$, as desired.
\end{proof}

\begin{remark}
    The ideas of the algorithm $\textsf{ZuckermanCount10}$ and of the proof of Theorem~\ref{thm:counting-complexity-10} might generalize to other composite values of $b$. 
\end{remark}

\begin{remark}\label{rmk:involved-max}
    Let $\gamma^\prime := \gamma \log 10 \log 2 /\! \log 8 \log 5$.
    To compute the maximum of~\eqref{equ:involved-max} under the aforementioned constraints, one can reason as follows.
    First, assuming that $y = \gamma^\prime$, one can compute the maximum under the constraints $\sum_{d=1}^9 (\log d) x_d \leq \alpha \log 10$ and $x_5 = 0$.
    This can be done by using Lemma~\ref{lem:restricted-digits} after the change of variables $x_d^\prime := x_d / (1-\gamma^\prime)$.
    This maximum is $0.359{\scriptstyle\ldots}$, which is less than $\gamma$.
    Then it remains to compute the maximum of
    \begin{equation*}
        F := -\frac1{2\log 10} \sum_{d \,=\, 1}^9 x_d \log\!\left(\frac{x_d}{S}\right) ,
    \end{equation*}
    where $S := \sum_{d = 1}^9 x_d$, under the constraints $x_5 = 0$,
    \begin{equation*}
        \frac{\log 2}{\log 10} (x_2 + 2x_4 + x_6 + 3x_8) = 1 - S ,
    \end{equation*}
    and $\sum_{d=1}^9 (\log d) x_d \leq \alpha \log 10$.
    This amount to solving the nonlinear system of equations given by the method of Lagrange multipliers, and can be done with arbitrary precision using numerical methods.
    The maximum is $0.3793709{\scriptstyle\ldots}$, which is less than $\gamma$.
    Note that the domain determined by the constraints is convex and $F$ is concave, since the Hessian matrix of $-F$ is positive semidefined.
    Thus each local maximum of $F$ is in fact a global maximum~\cite[p.~192, Theorem~1]{MR3363684}.
\end{remark}

\FloatBarrier

\begin{figure}[h]
    \centering
    \includegraphics{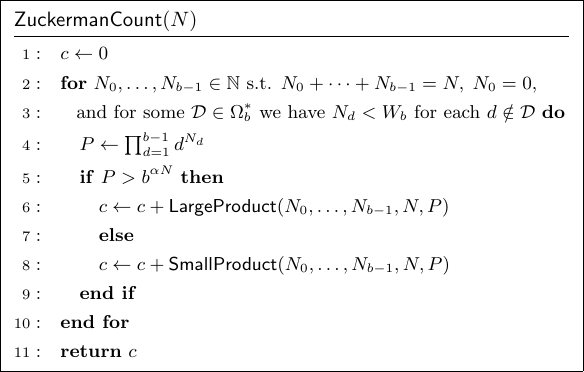}
    \caption{Algorithm to count the number of Zuckerman numbers with $N$ digits.}
    \label{fig:zuckerman-count}
\end{figure}

\begin{figure}[h]
    \centering
    \includegraphics{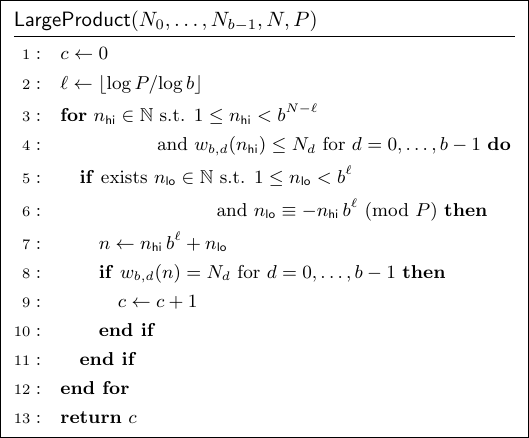}
    \caption{Subroutine $\textsf{LargeProduct}$ of $\textsf{ZuckermanCount}$.}
    \label{fig:large-product}
\end{figure}

\begin{figure}[h]
    \centering
    \includegraphics{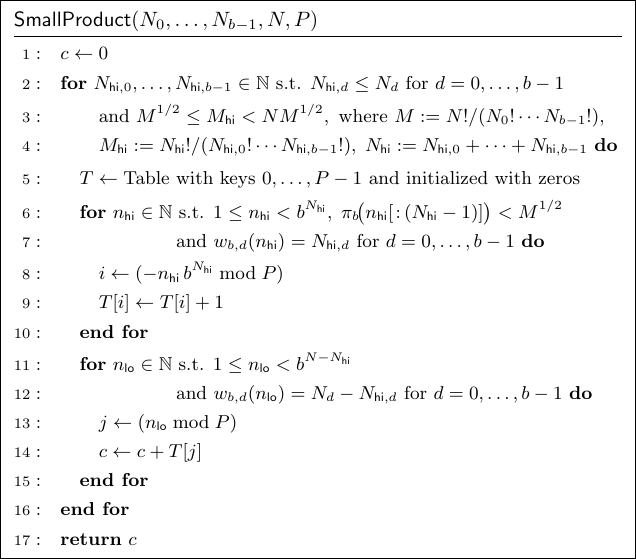}
    \caption{Subroutine $\textsf{SmallProduct}$ of $\textsf{ZuckermanCount}$.}
    \label{fig:small-product}
\end{figure}

\begin{figure}[h]
    \centering
    \includegraphics{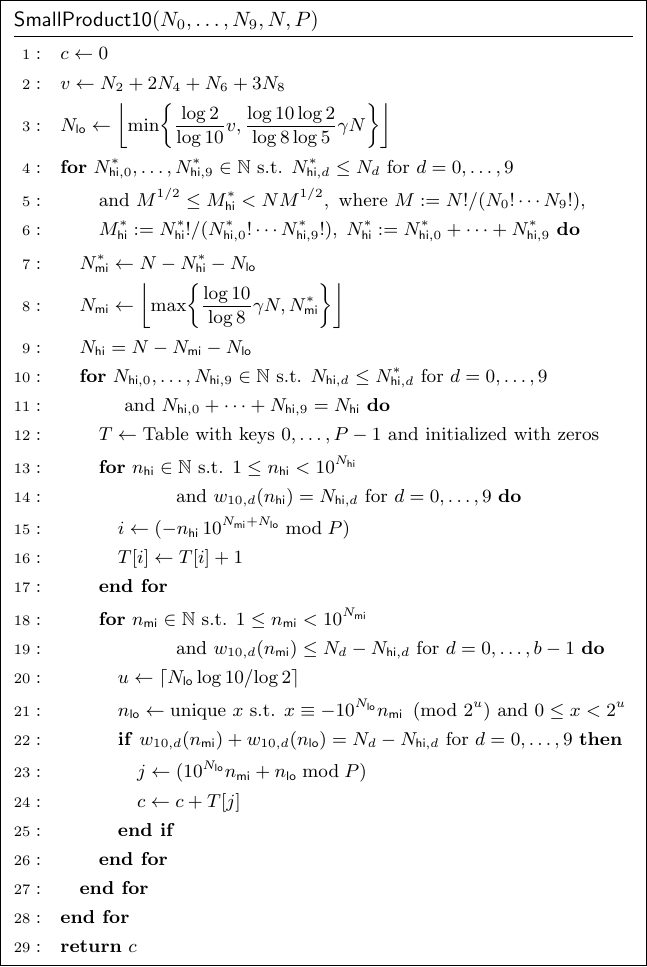}
    \caption{Subroutine $\textsf{SmallProduct10}$ of $\textsf{ZuckermanCount10}$.}
    \label{fig:small-product-10}
\end{figure}

\FloatBarrier
\bibliographystyle{amsplain-no-bysame}
\bibliography{main}

\end{document}